\documentclass[a4paper,12pt,reqno]{amsart}
\usepackage{amsfonts}
\usepackage{amsmath,amsthm,amssymb}

\nonstopmode \numberwithin{equation}{section}

\newtheorem{theorem}{Theorem}
\newtheorem{corollary}{Corollary}

\allowdisplaybreaks

\begin{document}
\title[Fractional Kinetic Equations Involving $k$-Bessel
Function]{Generalized Fractional Kinetic Equations Involving the generalized modified $k$-Bessel function}

\author[K. S. Nisar and J. Choi]{Kottakkaran Sooppy Nisar and Junesang Choi}
\address{K. S. Nisar : Department of Mathematics, College of Arts and
Science, Prince Sattam bin Abdulaziz University, Wadi Aldawaser, Riyadh
region 11991, Saudi Arabia}
\email{ksnisar1@gmail.com }
\address{J. Choi: Department of Mathematics, Dongguk University, Gyeongju
38066, Republic of Korea}
\email{junesang@mail.dongguk.ac.kr}

\begin{abstract}
Fractional kinetic equations are investigated in order to describe the
various phenomena governed by anomalous reaction in dynamical
systems with chaotic motion. Many authors have provided solutions
of various families of fractional kinetic equations involving
 special functions. Here, in this paper, we aim at presenting
solutions of  certain general families of fractional kinetic equations associated with the
generalized modified $k$-Bessel function of the first kind.
It is also pointed out that the main results presented here are general enough
to be able to be specialized to yield many known and (presumably) new
  solutions for fractional kinetic equations.
\end{abstract}

\maketitle

\section{Introduction, Notations and Preliminaries}\label{sec-1}

The Bessel function acting as a strong tool for investigating the solutions of
various types of differential equations has attracted a large numbers of researchers
such as mathematicians, physicists and engineers, due mainly to its importance
in mathematical physics, nuclear physics, systems and control theory (see \cite{Korenev}).
Various extensions and modifications of the Bessel function have been given and are
involved in solutions of fractional differential equations. From a statistical
point of view, the Bessel function is an unavoidable candidate to study various
types of  distribution theories. For more works of the Bessel function   related to statistics,
one may refer to \cite{Mathai-Stat}. The generalized Bessel function of the first kind $
w_{p}\left( z\right) $  is defined for $z\in
\mathbb{C} \setminus \{0\}$ and $b$, $c$, $p \in \mathbb{C}$ with   $\Re\left( p\right) >-1$ by
the following series representation (see  \cite{Baricz1}):
\begin{equation}\label{gbf-1} %%%(1.1)
w_{p,b,c}\left( z\right) =w_{p}\left( z\right) =\sum_{n=0}^{\infty }\,\frac{
\left( -1\right) ^{n}c^{n}(z/2)^{2n+p}}{n!\,\Gamma (p+\frac{b+1}{2}+n)},
\end{equation}
where $\mathbb{C}$ is the set of complex numbers and
$\Gamma \left(z\right)$ is the familiar Gamma function
whose Euler's integral is given by
 (see, \emph{e.g.}, \cite[Section 1.1]{Sr-Ch-12}):
 \begin{equation}\label{EGF} %%%(1.2)
\Gamma (z) = \int_0^\infty\, e^{-t}\,t^{z-1}\,dt \quad (\Re(z)>0).
\end{equation}
  More details of $w_{p}\left( z\right) $  can be found in the
recent works \cite{Baricz2, Choi-A-P, Saiful1}.

We consider three special cases of \eqref{gbf-1}.
\begin{itemize}
  \item  The special case of \eqref{gbf-1} when  $b=c=1$
  is easily seen to  reduce to the Bessel function of the first kind of order $p$
  which is defined by the series following representation (see, \emph{e.g.}, \cite{Baricz1, Watson}):
  \begin{equation} \label{Bff-1} %%%(1.2)
   J_{p}\left( z\right) =\sum_{n=0}^{\infty }\,\frac{\left( -1\right)^{n}(z/2)^{2n+p}}{n!\,\Gamma (p+n+1)}
   \quad \left(z,\, p\in \mathbb{C}; \,  \Re\left( p \right) >-1 \right).
\end{equation}

  \item Setting $b=1$ and $c=-1$ in \eqref{gbf-1} yields the following modified
     Bessel function of the first kind of order $p$ (see, \emph{e.g.}, \cite[p. 77]{Watson}):
  \begin{equation}\label{mbf-1} %%%(1.3)
   I_{p}\left( z\right) =\sum_{n=0}^{\infty }\,\frac{(z/2)^{2n+p}}{n!\,\Gamma
     (p+n+1)}\quad \left(z,\, p\in \mathbb{C}; \,  \Re\left( p \right) >-1 \right).
\end{equation}

  \item The special case of \eqref{gbf-1} when $b=2$ and $c=1$
      gives the following spherical Bessel function of the first kind
  (see, \emph{e.g.}, \cite[p. 77]{Watson}):
\begin{equation}\label{sbf-1} %%%(1.4)
j_{p}\left( z\right) =\frac{\sqrt{\pi }}{2}\sum_{n=0}^{\infty }\,\frac{
\left( -1\right) ^{n}(z/2)^{2n+p}}{n!\,\Gamma (p+n+3/2)}
\quad \left(z,\, p\in \mathbb{C}; \,  \Re\left( p \right) >-3/2 \right).
\end{equation}
\end{itemize}
In the above three cases, unless other restriction is given,
 it is assumed that the principal value of $\arg z$ is taken.

\vskip 3mm
The generalized Bessel function $\varphi_{p,b,c}\left( z\right) $
is given by the following transformation (see \cite{Deniz}):
\begin{eqnarray}\label{gbft-1}%%%(1.5)
\varphi _{p,b,c}\left( z\right) &=&2^{p}\Gamma \left( p+\frac{b+1}{2}\right)
z^{1-p/2}w_{p}\left( \sqrt{z}\right)  \notag \\
&=&z+\sum_{n=1}^{\infty }\frac{(-c)^{n}z^{n+1}}{n!\,4^{n}\left( \gamma
\right) _{n}},
\end{eqnarray}
where $\gamma = : p+\left( b+1\right) /2  \in \mathbb{C}\setminus \mathbb{Z}_{0}^{-}$,
 and $(\lambda )_{n}$ is the familiar Pochhammer symbol defined (for $\lambda \in
\mathbb{C}$) by (see, \emph{e.g.},  \cite[p.~2 and p.~5]{Sr-Ch-12}):
\begin{equation}\label{Pochhammer symbol}  %%%{1.6}
\aligned (\lambda)_n :
 & =\left\{\aligned & 1  \hskip 44 mm (n=0) \\
        & \lambda (\lambda +1) \ldots (\lambda+n-1) \quad (n \in {\mathbb N})
   \endaligned \right. \\
&= \frac{\Gamma (\lambda +n)}{ \Gamma (\lambda)}
      \quad \left(\lambda \in {\mathbb C} \setminus {\mathbb Z}_0^-\right),
\endaligned
\end{equation}
 $\mathbb{R}$ and $\mathbb{R}^+$,   $\mathbb{N}$ and $\mathbb{Z}^{-}_0$
being the sets of real and positive real numbers,   positive and non-positive integers, respectively, and  $\mathbb{N}_{0}:=\mathbb{N}\cup \{0\}$.

\vskip 3mm

The $k$-Bessel function of the first kind is defined by the following series
(see \cite{Romero-Cerutti}):
\begin{equation}\label{kbf-1} %%%{1.7}
J_{k,\mu }^{\gamma ,\lambda }\left( z\right)
=\sum_{n=0}^{\infty }\frac{\left( \gamma \right) _{n,k}}{\Gamma _{k}\left(
\lambda n+\mu+1\right) }\frac{\left( -1\right) ^{n}\left( z/2\right)
^{n}}{\left( n!\right) ^{2}}
\end{equation}
$$ \left(k\in \mathbb{R};\,\, \gamma, \,\lambda,\,  \mu \in \mathbb{C};\,\,
  \min\{\Re\left( \lambda \right),\,\Re\left(\mu \right)\}>0 \right),$$
where $\left(\gamma \right)_{n,k}$ is the $k$-Pochammer symbol  defined as follows
(see \cite{Diaz}):
\begin{equation}\label{kPochammer-1}%%%{1.8}
\left( \gamma \right)_{n,k}=\gamma \left( \gamma+k\right) \left( \gamma+2k\right) \cdots \left(
\gamma+\left( n-1\right) k\right) \quad (\gamma \in \mathbb{C};\,\, n\in \mathbb{N})
\end{equation}
and $\Gamma _{k}(z)$ is the $k$-Gamma function  defined by (see \cite{Diaz}):
\begin{equation}\label{kgamma-1}%%%{1.9}
\Gamma _{k}\left( z\right) =\int_{0}^{\infty }e^{-\frac{t^{k}}{k}
}t^{z-1}dt \quad \left(k \in \mathbb{R}^+;\,\, \Re\left( z\right) >0\right).
\end{equation}
It is easy to see that the $\Gamma _{k}(z)$ in \eqref{kgamma-1}
with $k=1$  reduces to the classical Gamma function
 $\Gamma \left(z\right)$.
Also the $\Gamma _{k}$ satisfies the following relations
(see \cite{Cerutti}):
\begin{equation} \label{kgamma-2}%%%{1.10}
\Gamma _{k}\left( z+k\right) =z\Gamma _{k}\left( z\right);
\end{equation}
\begin{equation}\label{kgamma-3}%%%{1.11}
\Gamma _{k}\left( z\right) =k^{\frac{z}{k}-1}\Gamma \left( \frac{z}{k}\right);
\end{equation}
\begin{equation}  \label{kPochammer-2}%%%{1.12}
\left( \gamma\right) _{n.k}=\frac{\Gamma _{k}\left( \gamma +nk\right) }{\Gamma
_{k}\left( \gamma \right)}.
\end{equation}

\vskip 3mm
Nisar and Saiful \cite{Nisar-Saiful} introduced and defined the following new generalization of $k$-Bessel function
 which is called generalized modified $k$-Bessel function of the first kind:
\begin{equation}  \label{kgbf-1}%%%{1.13}
\mathcal{J}_{b,k,\mu }^{c,\gamma ,\lambda }\left( z\right) =\sum_{n=0}^{\infty
}\frac{\left( c\right) ^{n}\left( \gamma \right) _{n,k}}{\Gamma _{k}\left(
\lambda n+\mu +\frac{b+1}{2}\right) }\frac{\left( z/2\right)
^{\mu +2n}}{\left( n!\right) ^{2}}
\end{equation}
$$ \left(k \in \mathbb{R}^+;\,\,\gamma,\,\lambda,\,\mu,\, b,\, c\in \mathbb{C};
\,\, \min\{\Re\left( \lambda\right), \Re\left(\mu\right) \}>0 \right).$$

\vskip 3mm

Fractional calculus has found many demonstrated applications in extensive fields
of engineering and science such as electromagnetics, fluid mechanics,
electrochemistry, biological population models, optics, signal processing
and control theory. It has been used to model physical and engineering
processes that are found to be best described by fractional differential
equations. The fractional derivative models are used for  modeling
of those systems that require accurate modeling of damping. Recent studies
showed that the solutions of fractional order differential equations can
model real-life situations better, particularly, in reaction-diffusion type
problems. Due to its potential applicability to a wide variety of problems,
fractional calculus has been developed for applications in a wide range of mathematics, physics and
 engineering (see, \emph{e.g.}, \cite{Caponetto, JHChoi,  Kilbas1, Manuel, Saigo1}).

During the last several decades, fractional kinetic equations of different forms have
been widely used in describing and solving several important problems of physics and
astrophysics. Many researchers have investigated and derived the solutions of the
fractional kinetic equations associated with various types of special functions
(see, \emph{e.g.},  \cite{VBL1,VBL2, Choi-D, Chouhan,  Dutta,Gupta, Haubold,Nisar-Kinetic,Saxena4, Saxena1,Saxena2,Saxena3}).
 Motivated by a large number of the above-cited investigations on the fractional kinetic
 equation, in this sequel, we propose to investigate solution of a certain generalized fractional
kinetic equation associated with the generalized modified $k$-Bessel function of the first kind.
It is also pointed out that the main results presented here can include, as their special cases, solutions of
many fractional kinetic equations which are (presumably) new and known.

 \vskip 3mm
 Consider an arbitrary reaction characterized by a time-dependent quantity $N = N(t)$. It is possible to calculate the rate of change $\frac{dN}{dt}$ to be a balance between the destruction rate $\mathfrak d$ and the production rate $\mathfrak p$ of $N$, that is,  $dN/dt =-\mathfrak d+\mathfrak p$.
In general, through feedback or other interaction mechanism, destruction and production depend on the quantity $N$ itself, that is,
 $$\mathfrak d = \mathfrak d(N)\qquad \text{and} \qquad \mathfrak p= \mathfrak p(N).$$
This dependence is complicated, since the destruction or the  production at a time $t$ depends not only on $N(t)$, but also on the past history $N(\eta)\;\;(\eta < t)$ of the variable $N$. This may be formally represented by the following equation (see \cite{Haubold}):
\begin{equation}  \label{dN-1}
\frac{dN}{dt}= -\mathfrak d\left(N_t\right)+\mathfrak p\left(N_t\right),
\end{equation}
where $N_t$ denotes the function defined by
$$N_t\left(t^*\right)=N\left(t-t^*\right)\qquad ( t^*>0).$$
Haubold and Mathai \cite{Haubold} studied a special case of the equation  \eqref{dN-1} in the following form:
\begin{equation}   \label{dN-2}
\frac{dN_i}{dt}= -c_i\,N_i\left(t\right)
\end{equation}
with the initial condition that $N_i\left(t=0\right)=N_0$ is the number density of species $i$ at time $t=0$ and the constant $c_i>0$.
This is known as a standard kinetic equation. The solution of the equation \eqref{dN-2} is easily seen to be given by
\begin{equation}   \label{sec1eqn3}
N_i\left(t\right) = N_0\, e^{-c_{i}t}.
\end{equation}
Integration gives an alternative form of the equation \eqref{dN-2} as follows:
\begin{equation}   \label{Nt-1}
N\left(t\right) - N_0 = c\cdot\;_{0}D_{t}^{-1} N\left(t\right),
\end{equation}
where $\;_{0}D_{t}^{-1}$ is the standard integral operator and $c$ is a constant.

The fractional generalization of the equation \eqref{Nt-1}
is given as in the following form (see \cite{Haubold}):
\begin{equation}   \label{sec1eqn5}
N\left(t\right) - N_0 = c^{\nu}\, _{0}D_{t}^{-\nu} N\left(t\right),
\end{equation}
where $\;_{0}D_{t}^{-\nu}$ is the familiar Riemann-Liouville fractional integral operator (see, \emph{e.g.}, \cite{Kilbas1} and \cite{Mi-Ro}) defined by
\begin{equation}   \label{sec1eqn6}
\;_{0}D_{t}^{-\nu}\, f(t)=\frac{1}{\Gamma\left(\nu\right)}\int_{0}^{t}\left(t-u\right)^{\nu-1} f\left(u\right) du \qquad \big(\Re\left(\nu\right)>0\big).
\end{equation}

\vskip 3mm

Suppose that $f(t)$ is a real- or complex-valued function of the (time) variable $t>0$ and
$s$ is a real or complex parameter. The Laplace transform
of the function $f(t)$ is defined by
\begin{equation}  \label{Laplace-Transform}
\aligned
F\left(p\right)= \mathcal{L}\left\{f\left(t\right): p \right\}&=\int_{0}^{\infty} e^{-pt}\, f\left(t\right) dt \\
&= \lim_{\tau \rightarrow \infty}\,\int_0^\tau\, e^{-pt}\, f\left(t\right) dt,
\endaligned
\end{equation}
whenever the limit exits (as a finite number).

\vskip 3mm

 Since  Mittag-Leffler introduced the so-called Mittag-Leffler function  $E_{\alpha}(z)$ (see \cite{Mi-Le}):
 \begin{equation}  \label{ML-ft}
 E_\alpha(z):=\sum_{n=0}^\infty\, \frac{z^n}{\Gamma (\alpha n+1)}
  \quad (z,\, \alpha \in \mathbb{C};\,|z|<0,\, \Re(\alpha) \geq 0),
\end{equation}
 a large number of its extensions and generalizations
 have been presented. The following rather simpler extension is recalled
 (see \cite{Wima}):
  \begin{equation}  \label{W-ML-ft}
 E_{\alpha,\beta}(z):=\sum_{n=0}^\infty\, \frac{z^n}{\Gamma (\alpha n+\beta)}
  \quad (z,\, \alpha,\, \beta \in \mathbb{C};\,|z|<0,\,\min\{\Re(\alpha), \Re(\beta)\}>0).
\end{equation}

\section{Solution of Generalized Fractional Kinetic Equations}
\label{sec-2}

We investigate the solution of the generalized
fractional kinetic equations involving the generalized modified $k$-Bessel function of the first kind \eqref{kgbf-1}.

\vskip 3mm

\begin{theorem} \label{thm-1}
Let $e$, $t$, $k$, $\nu \in \mathbb{R}^+$. Also let $b$, $c$,  $\gamma$, $\lambda$, $\mu \in \mathbb{C}$
with $\Re(\mu)>-1$. Then the solution of the following generalized fractional kinetic equation:
\begin{equation}\label{T1-1}
N\left( t\right) -N_{0}\,\mathcal{J}_{b,k,\mu}^{c,\gamma ,\lambda }\left( t\right)
=-e^{\nu }\text{ }_{0}D_{t}^{-\nu }N\left( t\right)
\end{equation}
is given by
\begin{equation}\label{T1-2}
N\left( t\right) =N_{0}\sum\limits_{n=0}^{\infty }\,\frac{c^{n}\left( \gamma \right) _{n,k}\,\Gamma (\mu+2n+1)}{\left( n!\right)^{2}\,\Gamma _{k}\left(
\lambda n+\mu +\frac{b+1}{2}\right) }\,\left(\frac{t}{2}\right)^{\mu +2n}
\,E_{\nu,2n+\mu+1}\left(- e^{\nu }t^{\nu }\right),
\end{equation}
where $E_{\nu,2n+\mu+1}\left( \cdot \right) $ is the generalized Mittag-Leffler
function in \eqref{W-ML-ft}.
\end{theorem}

\begin{proof} We begin by recalling the Laplace transform of the
Riemann-Liouville fractional integral operator (see, \emph{e.g.}, \cite{Erdelyi, Srivastava}):
\begin{equation} \label{P1-1}
\mathcal{L}\left\{ _{0}D_{t}^{-\nu }f\left( t\right) ;p\right\}
  =p^{-\nu}\,F\left( p\right),
\end{equation}
where, just as in the definition \eqref{Laplace-Transform},
$$ F\left(p\right)= \mathcal{L}\left\{f\left(t\right): p \right\}.$$

Taking the Laplace transform of both sides of \eqref{T1-1}
and using \eqref{kgbf-1} and \eqref{P1-1}, we obtain
\begin{equation} \label{Thm1-pf-1}
\mathcal{N}(p) = N_0\, \int_0^\infty\, e^{-pt}\,\sum_{n=0}^{\infty
}\frac{ c^{n}\,\left( \gamma \right) _{n,k}}{\Gamma _{k}\left(
\lambda n+\mu +\frac{b+1}{2}\right) }\frac{\left( t/2\right)^{\mu +2n}}{\left( n!\right) ^{2}} \, dt - e^\nu\,p^{-\nu}\, \mathcal{N}(p),
\end{equation}
where
  $$ \mathcal{N}(p) = \mathcal{L}\left\{N(t);p\right\}. $$
Integrating the integral in \eqref{Thm1-pf-1} term by term,
which is guaranteed under the given restrictions, and using \eqref{EGF},
we get: For $\Re(p)>0$,
$$ \aligned
\left(1+ (e/p)^\nu\right)\,\mathcal{N}(p)
   = & N_0\, \sum_{n=0}^{\infty
}\frac{ c^{n}\left( \gamma \right) _{n,k}}{\Gamma _{k}\left(
\lambda n+\mu +\frac{b+1}{2}\right) }\frac{2^{-\mu -2n}}{\left( n!\right)^{2}}
\,\int_0^\infty\, e^{-pt}\,t^{\mu +2n}\,dt\\
= & N_0\, \sum_{n=0}^{\infty
}\frac{c^{n}\left( \gamma \right) _{n,k}}{\Gamma _{k}\left(
\lambda n+\mu +\frac{b+1}{2}\right) }\frac{2^{-\mu -2n}}{\left( n!\right)^{2}}
\,\frac{\Gamma (\mu+2n+1)}{p^{\mu+2n+1}}.
\endaligned $$
Expanding $\left(1+ (e/p)^\nu\right)^{-1}$ as a geometric series, we have: For $e<|p|$,
$$ \aligned
\mathcal{N}(p)
 = & N_0\, \sum_{n=0}^{\infty
}\frac{ c^{n}\,\left( \gamma \right) _{n,k}}{\Gamma _{k}\left(
\lambda n+\mu +\frac{b+1}{2}\right) }\frac{2^{-\mu -2n}}{\left( n!\right)^{2}}
\,\frac{\Gamma (\mu+2n+1)}{p^{\mu+2n+1}}
\, \sum_{r=0}^\infty\, (-1)^r\, (e/p)^{\nu r}.
\endaligned $$
Taking the inverse Laplace transform and using the following known formula:
\begin{equation} \label{KnownF-pf-1}
\mathcal{L}^{-1} \left\{p^{-\alpha}\right\}= \frac{t^{\alpha-1}}{\Gamma (\alpha)}
\quad (\Re(\alpha)>0),
\end{equation}
we obtain
$$ \aligned
 & N(t)= \mathcal{L}^{-1}\left\{\mathcal{N}(p)\right\}\\
& =  N_0\, \sum_{n=0}^{\infty}\frac{c^{n}\left( \gamma \right) _{n,k}\,\Gamma (\mu+2n+1)}{\left( n!\right)^{2}\,\Gamma _{k}\left(
\lambda n+\mu +\frac{b+1}{2}\right) }\,\left(\frac{t}{2}\right)^{\mu +2n}
\, \sum_{r=0}^\infty\, \frac{(-1)^r\, (e t)^{\nu r}}{\Gamma (\nu r + \mu +2n +1)},
\endaligned $$
which, upon using \eqref{ML-ft}, yields the desired result \eqref{T1-2}.
\end{proof}

\vskip 3mm
Since  the generalized modified $k$-Bessel function of the first kind \eqref{kgbf-1}
includes many known functions as its special cases (see Section \ref{sec-1}),
so does the result in Theorem \ref{thm-1}. Here, we give just one example. Setting $k=\gamma=\lambda=1$ and
replacing $c$ by $-c$ in the result in Theorem \ref{thm-1} with the notations in Section \ref{sec-1},
 we obtain a known result asserted by the following corollary (see  \cite[Eq. (18)]{Dinesh}).
\vskip 3mm

\begin{corollary} \label{cor-1}
Let $e$, $t$, $\nu \in \mathbb{R}^+$. Also let $b$, $c$,   $\mu \in \mathbb{C}$
with $\Re(\mu)>-1$. Then the solution of the following generalized fractional kinetic equation
involving the generalized Bessel function of the first kind \eqref{gbf-1}:
\begin{equation}\label{C1-1}
N\left( t\right) -N_{0}\,w_{\mu,b,c}\left( t\right)
=-e^{\nu }\text{ }_{0}D_{t}^{-\nu }N\left( t\right)
\end{equation}
is given by
\begin{equation}\label{C1-2}
N\left( t\right) =N_{0}\sum\limits_{n=0}^{\infty }\,\frac{(-c)^{n}\,\Gamma (\mu+2n+1)}{n!\,\Gamma\left(
 n+\mu +\frac{b+1}{2}\right) }\,\left(\frac{t}{2}\right)^{\mu +2n}
\,E_{\nu,2n+\mu+1}\left(- e^{\nu }t^{\nu }\right),
\end{equation}
where $E_{\nu,2n+\mu+1}\left( \cdot \right) $ is the generalized Mittag-Leffler
function in \eqref{W-ML-ft}.
\end{corollary}

\vskip 3mm
We also provide two more general results  than
that in Theorem \ref{thm-1}, which are given in Theorems \ref{thm-2} and \ref{thm-3}.
They can be proved in parallel with the proof of Theorem \ref{thm-1}.
So the details of their proofs are omitted.
\vskip 3mm

\begin{theorem} \label{thm-2}
Let $e$, $t$, $k$, $\nu \in \mathbb{R}^+$. Also let $b$, $c$,  $\gamma$, $\lambda$, $\mu \in \mathbb{C}$
with $\Re(\mu)>-1$. Then the solution of the following generalized fractional kinetic equation:
\begin{equation}\label{T2-1}
N\left( t\right) -N_{0}\,\mathcal{J}_{b,k,\mu}^{c,\gamma ,\lambda }\left(e^\nu t^\nu \right)
=-e^{\nu }\text{ }_{0}D_{t}^{-\nu }N\left( t\right)
\end{equation}
is given by
\begin{equation}\label{T2-2}
N\left( t\right) =N_{0}\sum\limits_{n=0}^{\infty }\,\frac{c^{n}\left( \gamma \right) _{n,k}\,\Gamma (\mu+2n+1)}{\left( n!\right)^{2}\,\Gamma _{k}\left(
\lambda n+\mu +\frac{b+1}{2}\right) }\,\left(\frac{e^\nu t^\nu}{2}\right)^{\mu +2n}
\,E_{\nu,2n+\mu+1}\left(- e^{\nu }t^{\nu }\right),
\end{equation}
where $E_{\nu,2n+\mu+1}\left( \cdot \right) $ is the generalized Mittag-Leffler
function in \eqref{W-ML-ft}.
\end{theorem}

\vskip 3mm

\begin{theorem} \label{thm-3}
Let $a$, $e$, $t$, $k$, $\nu \in \mathbb{R}^+$. Also let $b$, $c$,  $\gamma$, $\lambda$, $\mu \in \mathbb{C}$
with $\Re(\mu)>-1$. Then the solution of the following generalized fractional kinetic equation:
\begin{equation}\label{T3-1}
N\left( t\right) -N_{0}\,\mathcal{J}_{b,k,\mu}^{c,\gamma ,\lambda }\left(e^\nu t^\nu \right)
=-a^{\nu }\text{ }_{0}D_{t}^{-\nu }N\left( t\right)
\end{equation}
is given by
\begin{equation}\label{T3-2}
N\left( t\right) =N_{0}\sum\limits_{n=0}^{\infty }\,\frac{c^{n}\left( \gamma \right) _{n,k}\,\Gamma (\mu+2n+1)}{\left( n!\right)^{2}\,\Gamma _{k}\left(
\lambda n+\mu +\frac{b+1}{2}\right) }\,\left(\frac{e^\nu t^\nu}{2}\right)^{\mu +2n}
\,E_{\nu,2n+\mu+1}\left(- a^{\nu }t^{\nu }\right),
\end{equation}
where $E_{\nu,2n+\mu+1}\left( \cdot \right) $ is the generalized Mittag-Leffler
function in \eqref{W-ML-ft}.
\end{theorem}

\section{Concluding Remarks} \label{sec-3}

The case $a=e$ in Theorem \ref{thm-3} reduces to the result in Theorem \ref{thm-2}.
The main results given in Section \ref{sec-2} are general enough
to be specialized to yield many new and known solutions of
the corresponding generalized fractional kinetic equations,
as in Corollary \ref{cor-1}. We conclude this paper by illustrating
such a special case of Theorem \ref{thm-3} as in the following corollary.

\vskip 3mm
\begin{corollary} \label{cor-2}
Let $a$, $e$, $t$, $\nu \in \mathbb{R}^+$. Also let  $\mu \in \mathbb{C}$
with $\Re(\mu)>-3/2$. Then the solution of the following generalized fractional kinetic equation involving
the  spherical Bessel function of the first kind in \eqref{sbf-1}:
\begin{equation}\label{C2-1}
N\left( t\right) -N_{0}\,j_\mu\left(e^\nu t^\nu \right)
=-a^{\nu }\text{ }_{0}D_{t}^{-\nu }N\left( t\right)
\end{equation}
is given by
\begin{equation}\label{C-2}
N\left( t\right) =N_{0}\sum\limits_{n=0}^{\infty }\,\frac{(-1)^{n}\,\Gamma (\mu+2n+1)}{\left( n!\right)^{2}\,\Gamma\left(
n+\mu +\frac{3}{2}\right) }\,\left(\frac{e^\nu t^\nu}{2}\right)^{\mu +2n}
\,E_{\nu,2n+\mu+1}\left(- a^{\nu }t^{\nu }\right),
\end{equation}
where $E_{\nu,2n+\mu+1}\left( \cdot \right) $ is the generalized Mittag-Leffler
function in \eqref{W-ML-ft}.
\end{corollary}

\begin{proof}
Setting $b=2$, $c=-1$, and $k=\lambda =\gamma =1$ in Theorem \ref{thm-3} and considering \eqref{sbf-1}
is easily seen to yield the desired result.

\end{proof}


\bigskip


\begin{thebibliography}{99}

\bibitem{Baricz1} \'{A}. Baricz, \emph{Generalized Bessel Functions of the First Kind},
    Lecture Notes in Mathematics, 1994, Springer, Berlin, 2010.



\bibitem{Baricz2} \'{A}.  Baricz, Geometric properties of generalized Bessel
functions, \emph{Publ. Math. Debrecen} \textbf{73} 1(2) (2008), 155--178.


\bibitem{Caponetto}  R. Caponetto, G. Dongola, L. Fortuna and I. Petras,
Fractional Order Systems, Modeling and Control Applications, World
Scientific, Series A, Vol. \textbf{72}, 2010.


\bibitem{Cerutti} R. A. Cerutti, On the $k$-Bessel functions,
\emph{Internat. Math. forum}  \textbf{7}(38) (2012), 1851--1857.



\bibitem{VBL1}  V. B. L. Chaurasia and Devendra Kumar, On the solutions of
 generalized fractional kinetic equations, \emph{Adv. Studies Theor. Phys.}
  \textbf{4}(16) (2010),  773--780.




\bibitem{VBL2}  V. B. L. Chaurasia and S. C. Pandey, On the new computable
solution of the generalized fractional kinetic equations involving the
generalized function for the fractional calculus and related functions,
\emph{Astrophys. Space Sci.} \textbf{317} (2008),   213--219.


\bibitem{JHChoi} J. H. Choi, Applications of multivalent functions
 associated with generalized fractional integral operator,
  \emph{Adv. Pure Math.} \textbf{3}(1) (2013), 5 pages;
  doi: 10.4236/apm.2013.31001.

\bibitem{Choi-A-P} J. Choi, P. Agarwal, S. Mathur, and S. D. Purohit,
Certain new integral formulas involving the generalized Bessel functions,
\emph{Bull. Korean Math. Soc.} \textbf{51}(4) (2014), 995--1003;
http://dx.doi.org/10.4134/BKMS.2014.51.4.995


\bibitem{Choi-D}  J. Choi and Dinesh Kumar, Solutions of generalized fractional kinetic equations involving
Aleph functions, \emph{Math. Commun.} \textbf{20} (2015), 113--123.





\bibitem{Chouhan}  A. Chouhan and S. Sarswat, On solution of generalized
kinetic equation of fractional order, \emph{Int. J.  Math. Sci. Appl.} \textbf{2}(2) (2012), 813--818.


\bibitem{Deniz} E. Deniz, H. Orhan, and H. M. Srivastava, Some sufficient
conditions for univalence of certain families of integral operators
involving generalized Bessel functions, \emph{Taiwanese J. Math.}
  \textbf{15}(2) (2011), 883--917.


\bibitem{Diaz} R. Diaz and E. Pariguan, On hypergeometric functions and $k$-Pochhammer symbol,
 \emph{Divulg. Mat.}  \textbf{15}(2) (2007), 179--192.



\bibitem{Dinesh} Dinesh Kumar, S. D. Purohit, A. Secer, and A. Atangana, On
generalized fractional kinetic equations involving generalized Bessel
function of the first kind,
\emph{Math. Probl. Eng.} (2016), Accepted for publication.


\bibitem{Dutta} B. K. Dutta, L. K. Arora and J. Borah, On the solution of
  fractional kinetic equation, \emph{Gen. Math. Notes} \textbf{6}(1) (2011), 40--48.


\bibitem{Erdelyi}  A.  Erd\'elyi, W.  Magnus,  F. Oberhettinger,
  and F. G. Tricomi,    {\it  Tables of Integral Transforms}, Vol. {\bf 1}
 McGraw-Hill Book Company, New York, Toronto, and London,  1954.




\bibitem{Gupta} V.G. Gupta, B. Sharma and F. B. M. Belgacem, On the solutions of generalized fractional kinetic
equations, Applied Mathematical Sciences, 5(19) (2011), 899-910.

\bibitem{Haubold}  H. J. Haubold and A. M. Mathai, The fractional kinetic
equation and thermonuclear functions, \emph{Astrophys. Space Sci.} \textbf{273} (2000),
53--63.


\bibitem{Kilbas1}
A. A. Kilbas, H. M. Srivastava and J. J. Trujillo,
 \emph{Theory and Applications of Fractional Differential Equations}, North-Holland Mathematics Studies  \textbf{204},
      Elsevier, Amsterdam, 2006.


\bibitem{Korenev} 	B. G. Korenev, \emph{Bessel Functions and Their Applications}, Taylor and Francis,
11 New fetter lane, London, 2002.

\bibitem{Mathai-Stat} A. M. Mathai and H. J. Haubold, Stochastic processes via
the pathway model, \emph{Entropy} \textbf{17} (2015), 2642--2654; doi:10.3390/e17052642

\bibitem{Mi-Ro} K. S. Miller and B. Ross, \emph{An Introduction to Fractional
              Calculus and Fractional Differential Equations}, A Wiley-Interscience Publication, John Wiley and Sons,
         New York, Chichester, Brisbane, Toronto and Singapore, 1993.

\bibitem{Mittag} G. M. Mittag-Leffler,  Sur la representation analytiqie d'une fonction monogene (cinquieme note),
    \emph{Acta Math.} \textbf{29} (1905), 101--181.



\bibitem{Mi-Le}
G. M. Mittag-Leffler, Sur la nouvelle function $E_{\alpha}(x)$,
\emph{C. R. Acad. Sci. Paris} {\bf 137} (1903), 554--558.


\bibitem{Saiful1} S. R. Mondal and K. S. Nisar, Marichev-Saigo-Maeda
fractional integration operators involving generalized Bessel functions,
\emph{Math. Probl. Eng.} (2014), Article ID 274093.



\bibitem{Nisar-Saiful} K. S. Nisar and S. R. Mondal, Certain unified integral
formulas involving the generalized modified $k$-Bessel function of the first kind,
arXiv:1601.06487 [math.CA], 2016.

\bibitem{Nisar-Kinetic} K. S. Nisar, S. D. Purohit, and S. R. Mondal, Generalized
fractional kinetic equations involving generalized Struve function of the
first kind, J. King Saud University-Sci. 2015.



\bibitem{Manuel}  M. D. Ortigueira and J. J. Trujillo, A unified
approach to fractional derivatives, \emph{Commun. Nonlinear Sci. Numer. Simulat.}
 \textbf{17} (2012), 5151--5157.


\bibitem{Romero-Cerutti} L. G. Romero, G. A. Dorrego and R. A. Cerutti, The
$k$-Bessel function of the first kind,
  \emph{Internat. Math. forum}  \textbf{7}(38) (2012), 1859--1864.


\bibitem{Saigo1}  M. Saigo and N. Maeda, More generalization of fractional
calculus, in \emph{Transform Methods $\&$ Special Functions}, Varna '96, 386--400,
Bulgarian Acad. Sci., Sofia.


\bibitem{Saxena4}  R. K. Saxena and S. L. Kalla, On the solutions of certain
fractional kinetic equations, \emph{Appl. Math. Comput.} \textbf{199} (2008), 504--511.


\bibitem{Saxena1}  R. K. Saxena, A. M. Mathai, and H. J. Haubold, On fractional
kinetic equations, \emph{Astrophys. Space Sci.} \textbf{282} (2002), 281--287.



\bibitem{Saxena2} R. K. Saxena, A. M. Mathai, and H. J. Haubold, On generalized
fractional kinetic equations, \emph{Physica A} \textbf{344} (2004), 657--664.



\bibitem{Saxena3}  R. K. Saxena, A. M. Mathai and H. J. Haubold, Solution of
generalized fractional reactiondiffusion equations, \emph{Astrophys. Space Sci.}
\textbf{305} (2006), 305--313.


\bibitem{Spiegel} M. R. Spiegel, \emph{Theory and Problems of Laplace Transforms},
Schaums Outline Series, McGraw-Hill, New York, 1965.


\bibitem{Sr-Ch-12}
H. M. Srivastava and J. Choi, \emph{Zeta and $q$-Zeta Functions and
Associated Series and Integrals}, Elsevier Science Publishers, Amsterdam,
London and New York, 2012.


\bibitem{Srivastava} H. M. Srivastava and R. K. Saxena, Operators of
fractional integration and their applications,
  \emph{ Appl. Math. Comput.} \textbf{118} (2001), 1--52.


\bibitem{Watson} G. N. Watson, \emph{A Treatise on the Theory of Bessel Functions},
          Cambridge University Press, Cambridge,   1944.



\bibitem{Wima} A. Wiman, \"Uber den fundamental satz in der theorie der funktionen $E_{\alpha}(x)$,
  \emph{ Acta Math.} \textbf{29} (1905), 191--201.


\bibitem{Zaslavsky} G. M. Zaslavsky, Fractional kinetic equation for
Hamiltonian chaos, \emph{Physica D} \textbf{76} (1994), 110--122.











\end{thebibliography}
\end{document}